\documentclass[10pt,a4paper]{article}
\usepackage[T1]{fontenc}
\usepackage{amsmath,amssymb,graphicx,amsthm}
\usepackage{pgf,pgfarrows,pgfnodes,pgfautomata,pgfheaps,pgfshade}
\usepackage{placeins}
\usepackage{float}
\usepackage{xcolor}
\usepackage{amsmath,amssymb}
\usepackage[numbers,sort&compress]{natbib}
\usepackage{colortbl}
\usepackage{comment}
\usepackage{multirow}
\usepackage{caption}
\usepackage[english]{babel}
\usepackage{graphicx,times}
\usepackage{bibentry}
\usepackage{xcolor}
\usepackage{subcaption}
\usepackage{hyperref}
\usepackage{authblk}
\usepackage[utf8]{inputenc}
\usepackage{mathrsfs}
\usepackage{textcomp}
\usepackage{manyfoot}
\usepackage{booktabs}
\usepackage{algorithm}
\usepackage{algorithmicx}
\usepackage{algpseudocode}
\usepackage{listings}

\newtheorem{The}{Theorem}[section]

\newtheorem{Def}{Definition}[section]

\FloatBarrier
\title{Analysis of a Class of Two-delay Fractional Differential Equation}
\author{Sachin Bhalekar\footnote{Corresponding author}}
\author{Pragati Dutta}
\affil{School of Mathematics and Statistics, University of Hyderabad, Hyderabad 500046, India }

\affil[ ]{\textit{Email addresses:} sachinbhalekar@uohyd.ac.in, pragati.dutta2617@gmail.com}
\date{}
\begin{document}

	\maketitle
	
	\begin{abstract}
		The differential equations involving two discrete delays are helpful in modeling two different processes in one model. We provide the stability and bifurcation analysis in the fractional order delay differential equation $D^\alpha x(t)=a x(t)+b x(t-\tau)-b x(t-2\tau)$ in the $ab$-plane. Various regions of stability include stable (S), unstable (U), single stable region (SSR), and stability switch (SS). In the stable region, the system is stable for all the delay values. The region SSR has a critical value of delay that bifurcates the stable and unstable behavior. Switching of stable and unstable behaviors is observed in the SS region.
	\end{abstract}
	%	\keywords{}
	\section{Introduction}
	Differential equation is a popular tool used in the mathematical modeling of physical systems.
	If the order of derivative in such a model is a non-integer, it is called a fractional differential equation (FDE) \cite{podlubny1998fractional,das2011functional}. The basic theory of FDEs is developed in \cite{mainardi1996fundamental,babakhani2003existence,daftardar2004analysis,lakshmikantham2008basic,ahmad2010existence,xu2013existence}. Luchko and Gorenflo \cite{luchko1999operational} developed an operational method to find analytical solutions to the linear FDEs. Generalization of fractional order can be done in many ways. A few popular fractional order derivatives (FOD) are Grunwald-Letnikov, Riemann-Liouville, and Caputo \cite{podlubny1998fractional}. The fractional derivative operators are nonlocal, unlike the classical derivatives, and hence the properties and results may be complicated \cite{bhalekar2019can,li2007remarks}. Numerical methods to solve FDEs are also time-consuming. Diethelm et al. \cite{diethelm2002predictor} developed a predictor-corrector method to solve FDEs. Daftardar-Gejji et al. \cite{daftardar2014new} improved this method and developed a new predictor-corrector method. Approximate analytical solutions of FDEs on a smaller interval are given by Adomian decomposition method \cite{daftardar2005adomian,ray2005analytical}, new iterative method \cite{daftardar2006iterative}, homotopy perturbation method \cite{he2003homotopy} and so on. In his seminal work \cite{matignon1996stability}, Matignon proposed the stability of the fractional order dynamical systems. The results are further explored by Tavazoei and Haeri \cite{tavazoei2008chaotic}. Fractional order counterparts of the popular chaotic systems are well studied in the literature \cite{ahmad2003chaos,li2004chaos,chen2008nonlinear,daftardar2010chaos,petravs2011fractional,kaslik2012nonlinear}. It is observed that the fractional order systems can show chaos for the system dimension less than three, unlike their classical counterparts. Control methods to synchronize chaos are developed by Zhou and Li \cite{zhou2005synchronization}, Deng \cite{deng2007generalized}, Bhalekar and Daftardar-Gejji \cite{bhalekar2010synchronization}, Mart{\'\i}nez-Guerra and P{\'e}rez-Pinacho \cite{martinez2018advances} among others.
	
	Hilfer discussed fractional calculus (FC) applications in physics in \cite{hilfer2000applications}. Mainardi worked on the FC in the viscoelasticity \cite{mainardi2022fractional}. Magin \cite{magin2004fractional,magin2010fractional} presented various applications of FC in bioengineering.
	% delay, FDDE
	\par The time delay can also be used to model a memory in the system, like a fractional derivative. The delay differential equations (DDE) are the infinite dimensional dynamical systems with rich dynamics \cite{smith2011introduction,hale2006functional}. One can have chaos in a single scalar DDE, unlike its nondelayed counterpart \cite{lakshmanan2011dynamics,uccar2002prototype,mackey1977oscillation}. As we can expect the effect of past states on the present and future, the occurrence of the delay is natural. A variety of applications of DDEs are presented in \cite{smith2011introduction,rihan2021delay,ruan2006delay,roussel1996use,keane2017climate,culshaw2000delay}.
	If one combines an FDE and a DDE, then the resulting fractional-order delay differential equation (FDDE) may lead to an interesting dynamical system. These systems will be advantageous in the view of applications. At the same time, the analysis of such systems will be challenging. The characteristic equations of FDDEs are transcendental equations containing the terms of the form $\lambda^\alpha$ and $\exp(-\lambda\tau)$ making them complicated. Bhalekar \cite{bhalekar2016stability,bhalekar2013stability} presented the stability analysis of $D^{\alpha}x(t)=ax(t)+bx(t-\tau)$. The analysis is further explored by Bhalekar and Gupta in \cite{bhalekar2022stability,bhalekar2023can,bhalekar2024stability,gupta2024fractional}. Numerical schemes to solve FDDEs are developed by Bhalekar and Daftardar-Gejji \cite{bhalekar2011predictor} and Daftardar-Gejji et al. \cite{daftardar2015solving}. Chaos in FDDEs is studied in \cite{bhalekar2010fractional,bhalekar2012dynamical,yuan2013chaos,wang2011analysis}. Bhalekar et al. \cite{bhalekar2011fractional} worked on the fractional order bloch equation arising in NMR.
	\par The DDEs
	\begin{equation}
		\dot{x}(t)=f\left( x(t), x(t-\tau_1), x(t-\tau_2) \right) \label{2dden}
	\end{equation}
	or their linearization at an equilibrium point
	\begin{equation}
		\dot{x}(t)=a x(t)+b x(t-\tau_1)+c x(t-\tau_2) \label{2dde}
	\end{equation}
	involving two delays $\tau_1$ and $\tau_2$ are well-studied  dynamical systems. Beuter et al. \cite{beuter1993feedback,beuter1989complex} used these equations to model human neurological diseases. Piotrowska \cite{piotrowska2008hopf} mentioned that the two delays model the two cellular processes, viz. proliferation and apoptosis. Belair et al. \cite{belair1995age} studied erythropoiesis using these equations. Braddock and Driessche \cite{braddock1983two} employed them in the population dynamics. The stability of these equations is analyzed by Hale and Hunag \cite{hale1993global}, Braddock and Driessche \cite{braddock1983two}, Mahaffy, Joiner and Zak \cite{mahaffy1995geometric}, Li, Ruan and Wei \cite{li1999stability} among others. Nussbaum \cite{nussbaum1978differential} studied the periodic solutions in these equations.
	\par The above discussion shows the importance of FDDEs and systems of the form (\ref{2dden}) and (\ref{2dde}). This motivates us to study the fractional order generalizations
	\begin{equation}
		D^\alpha x(t)=f\left( x(t), x(t-\tau_1), x(t-\tau_2) \right) \label{2ddenf}
	\end{equation}
	and
	\begin{equation}
		D^\alpha x(t)=a x(t)+b x(t-\tau_1)+c x(t-\tau_2). \label{2ddef}
	\end{equation}
	Bhalekar \cite{bhalekar2019analysing} studied (\ref{2ddef}) with $a=0$.
	In general, it is challenging to analyze these equations. Therefore, we consider a particular case of the equation (\ref{2ddef}) by setting $c=-b$ and $\tau_2=2\tau_1$. The paper is organized as below:
	
	Preliminaries are discussed in Section \ref{prel}. Section \ref{MR} deals with the main results. Section \ref{stab} covers the stability in the 	$ab$-plane. Section \ref{IE} presents illustrative examples, followed by the conclusions in Section \ref{con}.

	\section{Preliminaries}	
	\label{prel}
	
	We present some basic definitions\cite{podlubny1998fractional,luchko1999operational,samko1993fractional} and a result in this section.
	\begin{Def}
		A real function $f(t), t > 0,$  is said to be in space
		$C_\beta,\; \beta\in \mathbb{R}$ if there exists a real number $p (>\beta)$ such
		that $f(t) = t^p f_1(t)$, where $f_1(t) \in C[0,\infty).$
	\end{Def}
	
	\begin{Def}
		A real function $f (t), t > 0$, is said to be in space $C_\beta^m   ,\; m \in \mathbb{N} \cup \{0\}$, if $f^ {(m)}\in C_\beta.$
		
	\end{Def}
	
	\begin{Def}
		Let $f \in C^\beta$ and $\beta \ge -1$, then the Riemann 
		Liouville integral of $f$ of order $\mu, \; \mu > 0$, is given by
		\begin{equation}
			I^\mu f(t)=\frac{1}{\Gamma(\mu)} \int_a^t (t-\tau)^{\mu-1} f(\tau)d\tau,\quad t>a.\nonumber
		\end{equation}    
	\end{Def}
	\begin{Def}
		The Caputo fractional derivative of $f ,\; f \in C^m_{-1},m\in \mathbb{N} $, is defined as
		\begin{eqnarray}
			D^\mu f (t)& = &\frac{d^m}{dt^m}f(t),\quad \mu=m \nonumber\\
			&=& I^{m-\mu}\frac{d^m}{dt^m}f(t),\quad m-1<\mu<m. \nonumber
		\end{eqnarray}
		Note that for $m-1<\mu \le m, \; m\in \mathbb{N},$\\
		\begin{equation}
			I^\mu D^\mu f(t)=f(t)-\sum_{k=0}^{m-1}\frac{d^k f(0)}{dt^k}\frac{t^k}{k !}. \nonumber
		\end{equation}
	\end{Def}
	
	\subsection{Basic result}
	\begin{The}{\cite{bhalekar2016stability}}
		Suppose $x^*$ is an equilibrium solution of the generalized delay differential equation $D^\alpha x(t)=g(x(t),x(t-\tau)),\;0<\alpha\le1$ and $a=\partial_1 g(x^*,x^*), \; b=\partial_2 g(x^*,x^*).$\\
		\begin{enumerate}
			\item If $b \in (-\infty,-|a|)$ then the stability region of $x^*$ in $(\tau,a,b)$ parameter space is located between the plane $\tau=0$ and
			\begin{equation*}
				\tau_1(0)=\frac{\arccos\left( \frac{\left( a\cos(\frac{\alpha \pi}{2}) +\sqrt{b^2-a^2\sin^2(\frac{\alpha \pi}{2})}\right) \cos(\frac{\alpha \pi}{2})-a}{b}\right)} {\left( a\cos(\frac{\alpha \pi}{2}) +\sqrt{b^2-a^2\sin^2(\frac{\alpha \pi}{2})}\right)^{\frac{1}{\alpha}}}       
			\end{equation*}
			The equation undergoes Hopf bifurcation at this value.
			\item If $b \in (-a,\infty)$ then $x^*$ is unstable for any $\tau \ge 0$.
			\item If $b \in (a,-a)$ and $a<0$ then $x^*$ is unstable for any $\tau \ge 0$.
		\end{enumerate}
	\end{The}

	\section{Main Results}  \label{MR}
	Consider the fractional delay differential equation 
	\begin{equation}
		D^{\alpha}x(t)=ax(t)+bx(t-\tau)-bx(t-2\tau)   
		\label{eq.1}
	\end{equation}
	where $D^\alpha$ is a Caputo fractional derivative of order $\alpha\in(0,1],\;\tau\ge 0$ is the delay.
	The characteristic equation of this system can be found using the Laplace transform \cite{deng2007stability} as below 
	\begin{equation}\label{oce}
		\lambda^\alpha=a+b \exp({-\lambda \tau}) -b \exp({-2\lambda \tau}).
	\end{equation}
	By multiplying the above equation with $\tau^\alpha$ to get a new equation with a less number of parameters, we get
	\begin{equation}\label{mce}
		\beta^\alpha=A+B \exp({-\beta}) -B \exp({-2\beta}),
	\end{equation}
	
	where $ \beta=\lambda \tau, \; A=a\tau^\alpha \; \text{and} \; B =b\tau^\alpha$.\\
	Since, $Re(\beta)=\tau Re(\lambda) \;\text{ and } \; \tau>0,$	
	the stability properties of (\ref{mce}) are the same as those of  (\ref{oce}).
	\subsection{Stability Analysis}
	We conduct the stability analysis using the same approach outlined in \cite{smith2011introduction}.\\
	The system is asymptotically  stable if all the roots $\lambda$ of (\ref{oce}) satisfy
	\begin{equation}\label{sc}
		Re(\lambda)<0\;. 
	\end{equation}
	If $\tau=0$, then the condition (\ref{sc}) simplifies to the form
	$a<0$ which gives $A<0$ using $A=a\tau^\alpha.$.\\
	Now, we take $\tau>0$ and discuss the properties of roots of (\ref{oce}). Note that the sign of $Re(\beta)$ will be the same as that of $Re(\lambda)$. Therefore, we analyze eq. (\ref{mce}) and provide the properties of eq. (\ref{oce}) using the relation $\beta=\lambda\tau$.\\
	We write $\beta=u+iv,\; u,v\in \mathbb{R}$ and find the expressions for the curves which are the boundaries of the stable region in the $AB$-plane.
	On the boundary of the stable region, the condition
	\begin{equation}
		\label{bc}
		Re(\beta)=0 
	\end{equation}   hold. If $\beta \in \mathbb{R}$, the boundary condition \ref{bc} will become $\beta=0$, which gives $A=0$ (by (\ref{mce})).\\
	So, $A=0$ is a branch of the boundary curve in the $AB$-plane. Since for $\tau=0$, $A<0$ is the stability region and $A=0$ is the boundary curve, so $A>0$, i.e., positive $A$-axis is the region of instability. In the following theorem, we prove that for every $\tau>0$, the system will be unstable in the first and fourth quadrants of the $AB$-plane.
	
	\begin{The}
		Whenever $a>0$, the fractional delay differential equation 
		\begin{equation*}
			D^{\alpha}x(t)=ax(t)+bx(t-\tau)-bx(t-2\tau)    
		\end{equation*}
		is unstable for all $\tau\ge0$.
	\end{The}
	
	\begin{proof} Consider the modified characteristic equation (\ref{mce}).

		Take $ P(\beta)=\beta^\alpha -A $ and $Q(\beta)=B(\exp(-\beta)-\exp(-2\beta)).$
		We prove that there always exists a positive real root of (\ref{mce}) whenever $a>0$, i.e., there always exists a $\beta>0$  s.t  $P(\beta)=Q(\beta)$, whenever $a>0$.\\
		
		We consider two cases:\\
		
		Case 1: Assume $a>0$ and $b>0$ (i.e first quadrant).\\
		If $a>0$ and $b>0$ then $A>0$ and $B>0.$ \\
		It can be easily seen that $P(A^{1/\alpha})=0$ and $P(0)=-A$.
		Since $P'(\beta)=\alpha \beta^{\alpha-1} > 0 $. Therefore, $P(\beta)$ is monotonic increasing. Also, $\lim_{\beta\rightarrow\infty }P(\beta)=\infty$.
		Hence, $Range(P)=[-A,\infty)$.\\
		Now, $Q(0)=0$ and $Q(\beta) >0, \; \forall \; \beta>0. $ 
		Also, we observe that	$Q'(\beta)=0 $ when $\beta=\log2$ and $Q''(\log2)=-0.5B<0 $. Hence, $Q(\beta)$ has a local maxima (say $Q_M$) and $\lim_{\beta\rightarrow\infty}Q(\beta)=0.$
		So, $Range(Q)=[0,Q_M].$\\
		We can conclude that $Range(Q)\subset Range(P)$ (see Figure \ref{Fig. 1}).\\
		\begin{figure}[H]
			\centering{
				\includegraphics[scale=0.55]{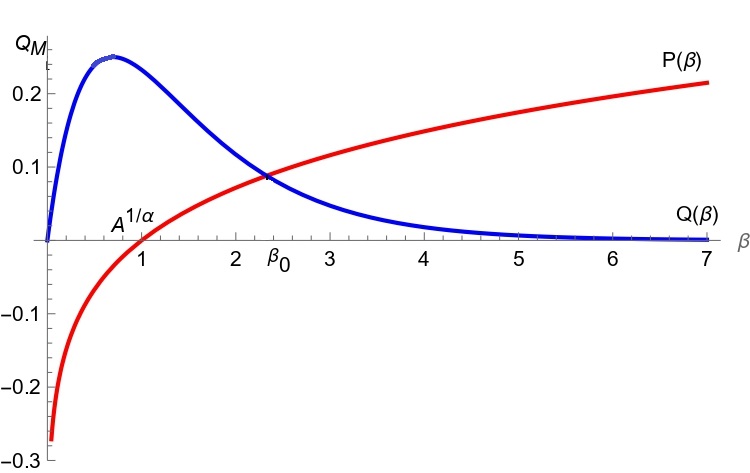} 
				\caption{P and Q intersect at some $\beta_0>0$}
				\label{Fig. 1}  }
		\end{figure}
		
		Hence, there always exist some $\beta_0>0$ s.t $P(\beta_0)=Q(\beta_0)$.
		This shows that there exists a real positive root to (\ref{mce}). So, the system is unstable in the first quadrant.\\
		
		Case 2: Assume $a>0$ and $b<0$ (i.e fourth quadrant).\\
		Therefore, we have $A>0$ and $B<0.$ 
		Since the function $P(\beta)$ is independent of $B$,
		the $Range(P)=[-A,\infty)$ in this case also.\\	 
		Now, $Q(0)=0$ and $Q(\beta) <0, \forall \; \beta>0. $ 
		Also, we observe that $Q'(\beta)=0 $ when $\beta=\log2$ and $Q''(\log 2)=-0.5B>0.$ Hence, in this case, $Q(\beta)$ has a local minima (say $Q_m$) and $\lim_{\beta\rightarrow\infty}Q(\beta)=0.$
		So, $Range(Q)=[Q_m,0].$\\
		We can conclude that $Range(Q)\cap Range(P) \neq \phi$ (see Figure \ref{Fig. 2}).\\
		\begin{figure}[H]
			\centering{
				\includegraphics[scale=0.5]{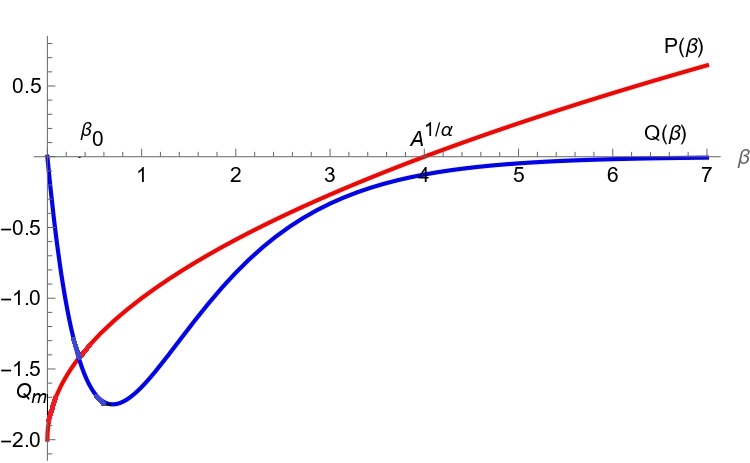}
				\caption{P and Q intersect at some $\beta_0>0$}
				\label{Fig. 2} }
		\end{figure}
		
		Hence, there always exist some $\beta_0>0$ s.t $P(\beta_0)=Q(\beta_0)$.	 Note that we can have one, two, or three such intersection points $\beta_0$, 
		This shows that $\exists$ a real positive root to (\ref{mce}). So, the system is unstable in the fourth quadrant also.

	\end{proof}
	
	\textbf{Note:} We denote the unstable region in the $ab$-plane by the letter 'U' in Figure \ref{fig:Fig. 5}.\\
	\vspace{0.2cm}\\
	Now, we analyze the stability in the second and third quadrants.  
	If $\beta \in \mathbb{C}$, then condition (\ref{bc}) gives $\beta=iv$. Since  $\overline{\beta}=-iv$ is also a root of (\ref{mce}), we may take $v>0$. By putting in $(\ref{mce})$ and solving for A and B in terms of parameter $v$, the following expressions can be obtained:
	\begin{eqnarray}
		A(\alpha, 
		v) &=& v^\alpha \cos\left({\frac{\alpha\pi}{2}}\right) + \left( \frac{\cos(2 v) - \cos(v)}{\sin(2 v) - \sin(v)}\right) v^\alpha \sin\left({\frac{\alpha\pi}{2}}\right) \label{A},\\
		B(\alpha, v) &=& v^\alpha \sin\left({\frac{\alpha\pi}{2}}\right)\left( \frac{\cos(2 v) - \cos(v)}{\sin(2 v) - \sin(v)}\right). \label{B}
	\end{eqnarray}

	In Figure \ref{fig:Fig. 3}, we present the parametric plots of $(A(\alpha,v), B(\alpha,v)),\; v\in (0,2\pi) $ for some values of $\alpha$ restricted to second and third quadrants of AB-plane. These restricted subintervals of $(0,2\pi)$ are given by\\
	\[I_1=\left[ (1-\alpha)\frac{\pi}{3},\frac{\pi}{3}\right) ,\;I_2=\left[ (5-\alpha)\frac{\pi}{3},\frac{5\pi}{3}\right)\;\text{and}\; I_3=\left[ (3-\alpha)\frac{\pi}{3},\pi\right).\]  \\
	
	Note that, at the left end of each of these subintervals, $A(\alpha,v)=0$ and at the right end, $\Gamma$ becomes an unbounded curve. If $\alpha$ is very small, there is no intersection between the first and second branches. 
	
\begin{figure}[H]
	\centering
	
	 \begin{subfigure}[b]{0.35\textwidth}
		\centering
		\includegraphics[width=\textwidth]{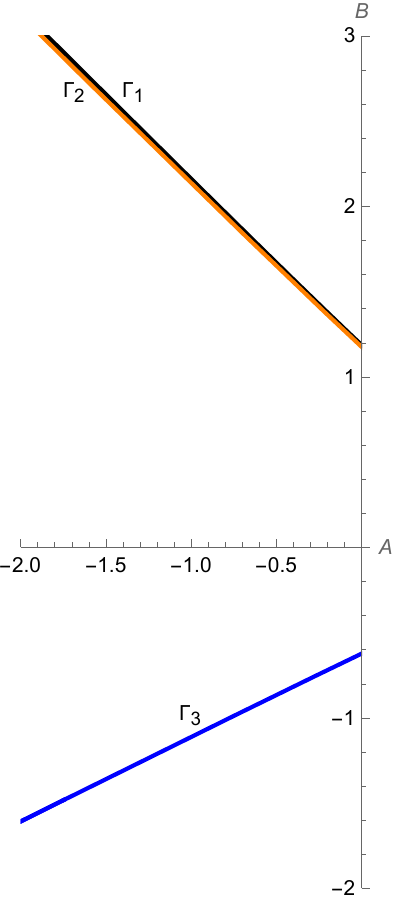} % Replace with your image
		\caption{$\alpha=0.2$}
		\label{fig:3(a)}
	\end{subfigure}
	\hspace{0.2cm}
	 \begin{subfigure}[b]{0.42\textwidth}
		\centering
		\includegraphics[width=\textwidth]{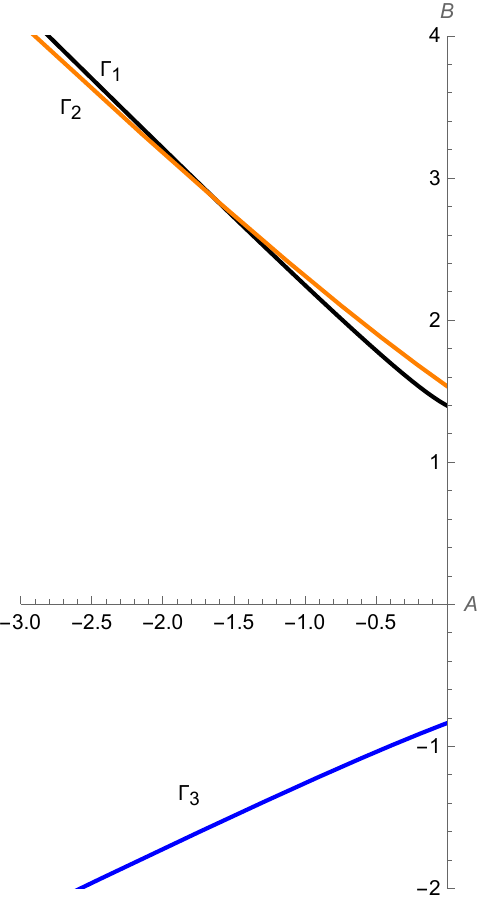} % Replace with your image
		\caption{$\alpha=0.5$}
		\label{fig:3(b)}
	\end{subfigure}
		\hspace{0.2cm}
	\begin{subfigure}[b]{0.38\textwidth}
		\centering
		\includegraphics[width=\textwidth]{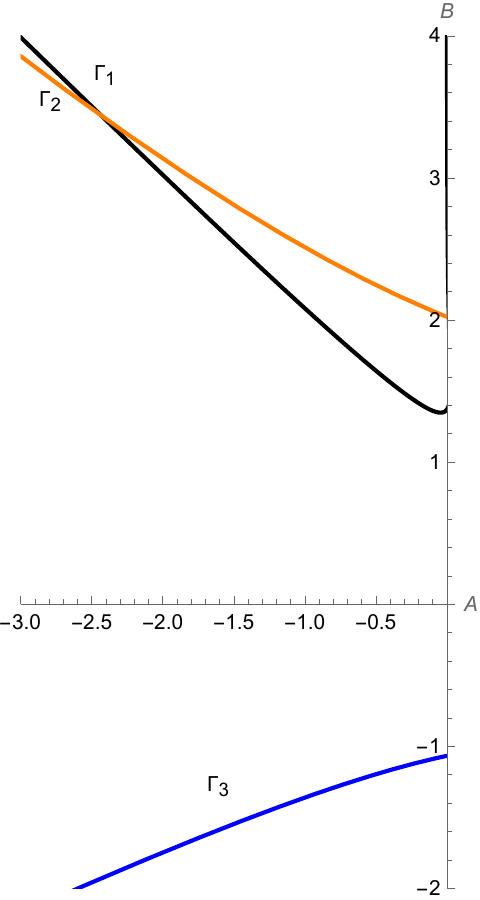} % Replace with your image
		\caption{$\alpha=0.8$}
		\label{fig:3(c)}
	\end{subfigure}
		\hspace{0.4cm}
	\begin{subfigure}[b]{0.38\textwidth}
		\centering
		\includegraphics[width=\textwidth]{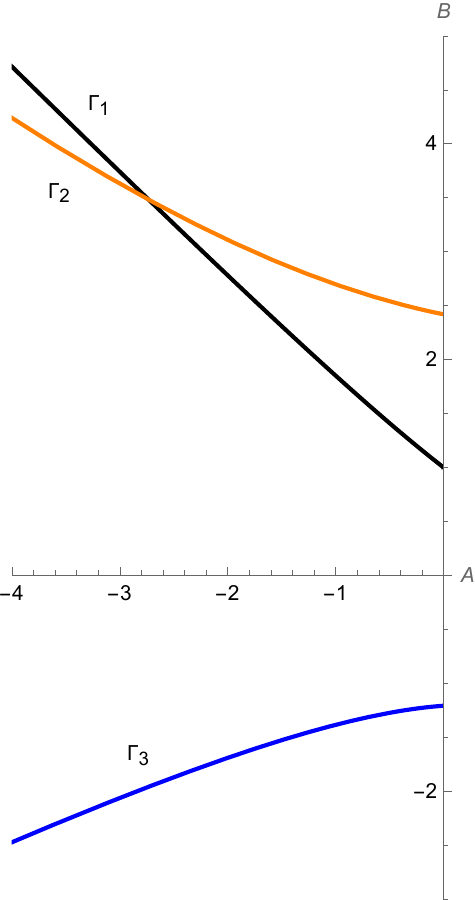} % Replace with your image
		\caption{$\alpha=1$}
		\label{fig:3(d)}
	\end{subfigure}
	\caption{Parametric plot of $(A(\alpha,v), B(\alpha,v))$ in second and third quadrants for different values of $\alpha$}
	\label{fig:Fig. 3}

\end{figure}

	\clearpage
	
	For $\tau=0$, the system is stable if $A<0$, which is the negative A-axis. After this, the stability will change at the branches of curves given in Figure \ref{fig:Fig. 3}.\\
	
	To get those critical values where the stability will change, we first find the value of $v \in I_1 \cup I_2 \cup I_3$ corresponding to the critical value of delay. Such values of $v$ can be determined by solving 
	\begin{equation}
		\frac{b}{a} - \left( \frac{1}{2} \csc\left(\frac{v}{2}\right) \sec\left(\frac{(3 v + \pi \alpha)}{2}\right)  \sin\left(\frac{\pi \alpha }{2}\right)\right) = 0. 
		\label{rel}
	\end{equation}
	
	Note that (\ref{rel}) is obtained by using the relations $A=a\tau^\alpha \; \text{and} \; B =b\tau^\alpha$  i.e.  $\frac{b}{a}=\frac{B}{A}$.\\
	Since this is a transcendental equation, it will have many roots.\\
	By putting different values of $\alpha$ and $(a, b)$ value in the respective quadrants, we get two values of $v$ in the second quadrant and one in the third.\\
	Now, in the second quadrant, we have two branches of the curve $\Gamma=(A, B)$ parametrized by (\ref{A}) and (\ref{B}), namely $\Gamma_1(\alpha, v)=(A(\alpha, v), B(\alpha, v)),\; v\in I_1
	\; \text{and}\; \Gamma_2(\alpha, v)=(A(\alpha, v),B(\alpha, v)),\; v\in I_2\; $,

	The boundary of the stable region in this quadrant is given by the part of the branch that is closest to the A-axis. We denote the boundary in this quadrant as $\Gamma_4$. \\
	In the third quadrant, there is only one branch of the curve $\Gamma=(A,B)$, namely $\Gamma_3(\alpha, v)=(A(\alpha, v), B(\alpha, v))$ for $v\in I_3$.
	\vspace{0.3cm}	\\
	\textbf{Observations 2.1}
	In the second quadrant, $\alpha^*=0.2644385$ bifurcates the two behaviors as follows (see Figure \ref{fig:bif pt}): 
	\begin{itemize}
		\item If $\alpha\leq\alpha^*$ then the second branch $\Gamma_2(\alpha, v)$ is closest to the $A$-axis and hence, it is the boundary of the stable region in the second quadrant (e.g., see Figure \ref{fig:3(a)}) i.e., $\Gamma_4=\Gamma_2.$ \\
		\item If $\alpha>\alpha^*$ then the curves $\Gamma_1(\alpha, v)$ and $\Gamma_2(\alpha, v)$ intersect at a point, say $(A_1(\alpha),B_1(\alpha))$ (e.g. see Figures \ref{fig:3(b)},\ref{fig:3(c)},\ref{fig:3(d)}). (see data set 1 for the code to find $A_1(\alpha),\;B_1(\alpha)$ numerically and Figure \ref{int} for the graph.) 
		
		\begin{figure}[!h]
			\centering
			\includegraphics[scale=0.5]{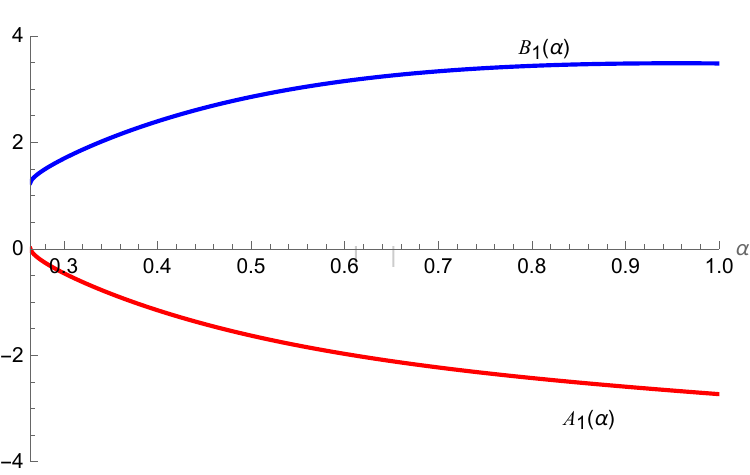}
			\caption{Curves $A_1(\alpha)$ and $B_1(\alpha)$}
			\label{int}
		\end{figure}

		\begin{center}
			\begin{itemize}
				\item 	If $A \le A_1(\alpha)$ then the boundary is given by $\Gamma_2(\alpha,v)$,\\
				\item   If $A>A_1(\alpha)$ then $\Gamma_1(\alpha, v)$ gives the boundary.  \\   
			\end{itemize}
		\end{center}
	\end{itemize}
	i.e., \begin{equation}
		\Gamma_4 = 
		\begin{cases}
			\Gamma_2, & \text{if}  A\le A_1 \\
			\Gamma_1, & \text{if}  A>A_1
		\end{cases}
	\end{equation}
	
	At the bifurcation point, the two branches of the boundary curve have an intersection on the B-axis. So, for $\alpha=\alpha^*, \; \exists \; v_1 \in I_1 \; \text{and} \; v_2 \in I_2$ such that
	\[ A(\alpha^*, v_1)=A(\alpha^*, v_2)=0 \; \text{and} \; B(\alpha^*, v_1)=B(\alpha^*, v_2) \;\; \text{(as shown in Figure \ref{fig:bif pt}).} \]
	As already discussed, $A(\alpha,\frac{(1-\alpha)\pi}{3})=A(\alpha,\frac{(5-\alpha)\pi}{3})=0.$ So, $v_1=\frac{(1-\alpha)\pi}{3} \; \text{and} \; v_2=\frac{(5-\alpha)\pi}{3}.$
	Therefore, we solve $B(\alpha,\frac{(1-\alpha)\pi}{3})=B(\alpha,\frac{(5-\alpha)\pi}{3})$ to get the approximate value of the bifurcation point viz. $\alpha^*=0.2644385.$\\

\begin{figure}[H]
	\centering
	
	\begin{subfigure}[b]{0.4\textwidth}
		\centering
		\includegraphics[width=\textwidth]{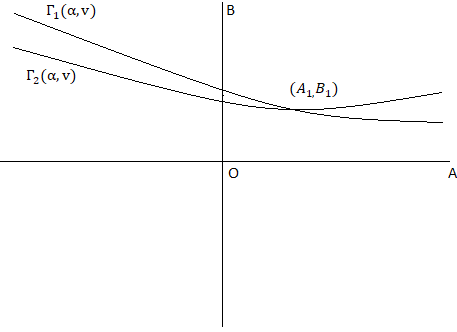} % Replace with your image
		\caption{$\alpha<\alpha^*$}
		\label{fig:bb}
	\end{subfigure}
	\hspace{0.2cm}
	\begin{subfigure}[b]{0.4\textwidth}
		\centering
		\includegraphics[width=\textwidth]{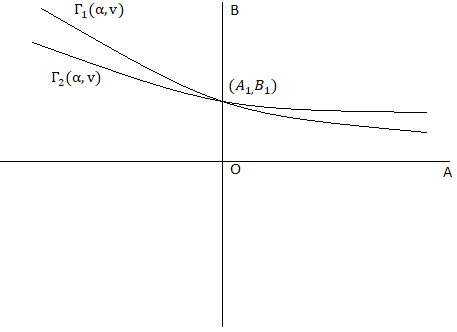} % Replace with your image
		\caption{$\alpha=\alpha^*$}
		\label{fig:ab}
	\end{subfigure}
	\hspace{0.2cm}
	\begin{subfigure}[b]{0.5\textwidth}
		\centering
		\includegraphics[width=\textwidth]{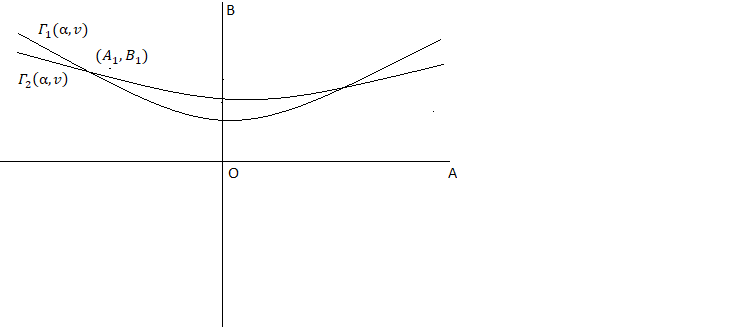} % Replace with your image
		\caption{$\alpha>\alpha^*$}
		\label{fig:afb}
	\end{subfigure}

	\caption{Plots showing behavior of bifurcation point $\alpha^*$}
	\label{fig:bif pt}
	
\end{figure}

	The stable region is bounded by the curves $\Gamma_4 \; \text{and} \; \Gamma_3$ closest to the negative $A-$ axis. The region outside the stable region is the unstable one. Thus, Figure \ref{fig:Fig. 3} show the stable and unstable regions for the modified characteristic equation (\ref{mce}). \\
	Note that the stability in the $ab$-plane depends on delay $\tau$ also. In the next section, we discuss such regions in the $ab$-plane.
	
	\section{Stability diagram in $ab$-plane} 
	\label{stab}
	In this section, we use the analysis in section \ref{MR} and provide various stable regions in the $ab$-plane for the given DDE (\ref{eq.1}).\\
	Recall, $A=a\tau^\alpha\; \text{and} \; B=b\tau^\alpha$. For any fixed $(a,b)$, we consider the vector 
	\[\overline{V}=\{(A, B)=(a\tau^\alpha, b\tau^\alpha)\;|\; \tau \ge 0\}.\]
	At $\tau=0,\; \overline{V}$ is at origin in the $AB$-plane. As we increase $\tau,\; \overline{V} $ represents an arrow.\\
	We have the following behaviors:
	\begin{enumerate}
		\item For given pair $(a, b),\; \overline{V}$ generates an arrow which lies entirely in the unstable region in the $AB$-plane. Such pairs $(a, b)$ form an unstable region in the $ab$-plane. 
		\item For given $(a, b),\; \overline{V}$ generates an arrow lying completely in the stable region in the $AB$-plane. This gives the stable region in the $ab$-plane.\\
		
		The two regions described above are independent of the delay.
		\item There are some pairs $(a, b)$ such that a part of $\overline{V}$ lies in the stable region and another part in the unstable one.\\
		e.g.,
		\begin{itemize}
			\item If $\exists\; \tau^*$ such that\\
			$0<\tau<\tau^*$ $\implies$ $\overline{V}\in$ stable region and \\
			$\tau>\tau^*$ $\implies$  $\overline{V}\in$ unstable region.\\
			Then, we get a single stable region (SSR) in the $ab$-plane.
			\item If $\exists \; \tau_1^*\; \text{and}\; \tau_2^* $ such that \\
			$0<\tau<\tau_1^*\; \text{and} \; \tau_2^*<\tau<\infty$ $\implies$ $\overline{V}\in$ stable region and \\
			$\tau_1^*<\tau<\tau_2^*$ $\implies$  $\overline{V}\in$ unstable region.\\
			Then, we get stability switch (SS) viz. S-U-S in the $ab$-plane.
		\end{itemize}
	\end{enumerate}
	The critical values $\tau^*,\;\tau_1^* \; \text{and} \; \tau_2^*$ are given by the intersection of $\overline{V}$ with $\Gamma_4 \; \text{or} \; \Gamma_3$ and can be expressed as 
	\[\left(\frac{A(\alpha, v)}{a}\right)^\frac{1}{\alpha}=\left(\frac{B(\alpha, v)}{b}\right)^\frac{1}{\alpha}, \;\; \text{for suitable A and B on the boundary $\Gamma_3 \; \text{or}\; \Gamma_4$}\] 
	(by using $A=a\tau^\alpha, B=b\tau^\alpha$).\\
	In this expression, if the boundary is given by $\Gamma_j$ then we can find $v\in I_j$ by solving \ref{rel}.\\
	Expressions for all the critical values are given in data set 2.
	The next section provides more details on these critical values and the SS and SSR regions.
	\subsection{Main Analysis}
	Note that, by using expressions $A=a\tau^\alpha \; \text{and} \; B=b\tau^\alpha$, we can convert the regions of stability in the  $AB$-plane to $ab$-plane. Both $A$-axis and $B$-axis will convert to $a$-axis and $b$-axis respectively in $ab$-plane. Also, any line passing through the origin in $AB$-plane, say $B=mA$,  will be converted to $b=ma$.\\
	
	In the second quadrant of the $AB$-plane, we observe that the boundary curve $\Gamma_4$ has a local minima but no local maxima.
	As a result, we can find a tangent $T_1$ (in the second quadrant) to $\Gamma_4$, which passes through the origin (denoted as $b=m_1(\alpha)a$). After applying the transformations we discussed, the region enclosed by the $T_1$ and the negative $A$-axis will represent the stable region in the $ab$-plane.
	If we consider a line $B=mA$ above $T_1$ where $m<m_1$ and sufficiently small $|m_1-m|$ then this line will intersect $\Gamma_4$ in two points leading to two critical points $\tau_1^*$ and $\tau_2^*$ in the SS-region of $AB$-plane. As we decrease $m$, the second intersection goes away from the origin and vanishes as we reach $m=-1$ (Further details are provided in the subsequent discussion).	Thus, the line $T_2$ given by $B=-A$ is another boundary $b=-a$ in the $ab$-plane.
	
	The region bounded by $T_1$ and $T_2$ in the $ab$-plane is the SS-region, where we get two critical values $\tau_1^*$ and $\tau_2^*$ discussed in the previous section. If we decrease $m$ further i.e., $-1>m>-\infty$ then the line $B=mA$ lies between $T_2$ and the vertical $B$-axis and cuts $\Gamma_4$ only at one point. This gives only one critical value $\tau^*$ of the delay and the region SSR in the $ab$-plane.\\ 
	Also, in the third quadrant, only one branch $\Gamma_3$ of the boundary exists. Let us denote by $T_3$, the tangent line to $\Gamma_3$, passing through the origin. We show that $T_3$ is given by $B=\frac{1}{2}A$ in the following discussion. \\
	So, the region bounded by $T_3$ and negative  $A$-axis will be the stable region and the region bounded by $T_3$ and negative $B$-axis will be the 
	SSR (Single Stable Region) in the $ab$-plane after the transformations we discussed (see Figure \ref{Fig. 4}).\\
	
	\begin{figure}[H]
		\centering{\includegraphics[scale=0.6]{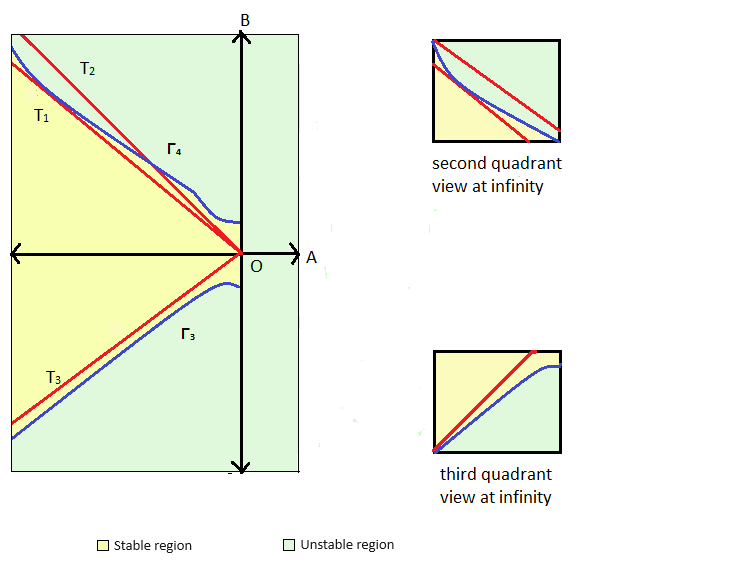}}
		\caption{Tangent lines $T_1,T_2\; \text{and} \; T_3$ (not to the scale)}
		\label{Fig. 4}
	\end{figure}
	
	\subsection{Expressions for $T_1,\; T_2 \; \text{and} \; T_3$}
	
	%	Now, we find the explicit equations of $T_1,\; T_2 \; \text{and} \; T_3$
	Note that, $T_1$ is tangent to $\Gamma_2$ and $T_3$ is tangent to $\Gamma_3$. The line $T_2$ has first intersection with $\Gamma_2$ and is tangent to $\Gamma_2$ at infinity.\\
	For fixed $\alpha\in(0,1],$ the equation of tangent to the curve $\Gamma_j(\alpha,v)=(A(\alpha,v),B(\alpha,v))$, touching at some point $\Gamma_j(\alpha,v_0)$, in parametric form is given by 
	\begin{equation}
		I(v)=\Gamma_j(\alpha,v_0)+(v-v_0)\frac{\partial\Gamma_j}{\partial v}(\alpha,v_0)  %refer some book for this%
	\end{equation}
	where $v_0 \in I_j, \; j=2,3$.

	This tangent line should also pass through the origin, i.e.,
	\begin{eqnarray}
		&&    I(v)=0,\;\; \text{for some $v$}.    \nonumber \\
		\vspace{0.01cm}       \nonumber\\ 
		&&  \Rightarrow \Gamma(v_0)+(v-v_0)\frac{\partial\Gamma_j}{\partial v}(\alpha,v_0)=0  \nonumber \\
		\vspace{0.02cm}        \nonumber\\ 
		&& \Rightarrow A(\alpha,v_0)+(v-v_0)\frac{\partial A}{\partial v}(\alpha,v_0)=0, \nonumber\\ 
		&&\;\;\;\; B(\alpha_0,v_0)+(v-v_0)\frac{\partial B}{\partial v}(\alpha,v_0)=0. \nonumber\\	
		\vspace{0.02cm}      \nonumber\\ 
		\text{ Eliminating $v$, we get}   \nonumber \\
		&&\frac{A(\alpha,v_0)}{\frac{\partial A}{\partial v}(\alpha,v_0)}=\frac{B(\alpha,v_0)}{\frac{\partial B}{\partial v}(\alpha,v_0)}.        \label{cft} 
	\end{eqnarray}

	After solving this equation for $v_0(\alpha)$, we find the slope of tangent line using 
	\begin{equation}
		m(\alpha)=\frac{dB}{dA}(\alpha,v_0)=-\frac{(\csc(\frac{v_0}{2}))^2(-v_0\cos v_0 + 2v_0 \cos 2v_0 + \alpha(\sin v_0 - \sin 2v_0))\sin(\frac{\alpha\pi}{2})}{2(\alpha \cos(\frac{\alpha\pi}{2})+\alpha \cos(3v_0+ \frac{\alpha\pi}{2})-3v_0\sin(\frac{\alpha\pi}{2}))}.
		\label{m}
	\end{equation}
	
	Now, by solving (\ref{cft}), we get many values of $v$, but we want the values of v in the intervals $I_2\cup \{\frac{5\pi}{3}\}$ and $I_3\cup \{\pi\}$.\\
	\begin{enumerate}
		\item  In $I_3\cup\{\pi\}$ (i.e., in the third quadrant), the only permissible value of $v$ is $v_0=\pi$ for all values of $\alpha$. By putting $v_0=\pi$ in (\ref{m}), we get $m=1/2.$ \\
		Hence, the equation of tangent line $T_3$ is $B=A/2.$\\
		Using the discussion in this section's beginning, this tangent becomes $b=a/2$ in $ab$-plane.
		\item  In $I_2\cup\{\frac{5\pi}{3}\}$ (i.e., in the second quadrant), we get a bifurcation value $\alpha^{**}=2/3 $ such that 
		\begin{itemize}
			\item For $\alpha\le2/3, \; v_0=\frac{5\pi}{3}$ is the only solution of (\ref{cft}). By using this value of $v_0$ in  (\ref{m}), we get $m_1=-1.$
			Hence, the equation of tangent line $T_1$ is $B=-A$.
			In the $ab$-plane, the corresponding line will be $b=-a.$
			\item For $\alpha>2/3, \; v_0=v_0(\alpha)<\frac{5\pi}{3}$, which is not constant. Therefore, $m_1(\alpha)$ is also not a constant.\\
			Hence, the equation of tangent line $T_1$ is $b=m_1(\alpha)a$.\\
			
			The expression for slope $m_1(\alpha)$ is provided in the data set 3.
		\end{itemize}
		\vspace{0.1cm}
		
		The bifurcation point $\alpha^{**}=2/3$ can be determined by equating $v_0(\alpha)=\frac{5\pi}{3}$ and solving for $\alpha$.
		
		\item Now, $T_2$ is tangent to $\Gamma_2$ at its unbounded end.\\
		Note that, $\Gamma_2$ becomes unbounded as $v\rightarrow\frac{5\pi}{3}$. Therefore, taking limit $v_0\rightarrow\frac{5\pi}{3}$ in (\ref{m}), we get $m(\alpha)=-1, \; \forall \; \alpha \in (0,1).$  Thus, the equation of $T_2$ is given by $B=-A$. $b=-a$ will be the corresponding line in the $ab$-plane.
	\end{enumerate}

	\begin{itemize}
		\item   So, $T_1$ and $T_2$ are the same for $\alpha\le2/3$. Hence, there will be no stability switch for these values of $\alpha.$

	\end{itemize}
	
	We present the stability diagrams in $ab$-plane in Figure \ref{fig:Fig. 5}, and summarize these results in Theorem \ref{Main Theorem}.
	\begin{The}
		\label{Main Theorem}
		Consider the FDDE 
		\begin{equation}
			D^\alpha x(t)=ax(t)+bx(t-\tau)-bx(t-2\tau), \;\; 0<\alpha\le1 \; \text{and}\; \tau\ge0.
		\end{equation}
		The zero solution of this equation is --
		\begin{enumerate}
			\item asymptotically stable $\forall \; \tau\ge0$ if
			\begin{itemize}
				\item $0<\alpha\le2/3, \; a<0$ and $a/2<b<-a$.
				\item $2/3<\alpha\le1, \; a<0$ and $a/2<b<m_1(\alpha)a$,
			\end{itemize} 
			where $m_1(\alpha)$ is given in data set 3 and sketched in Figure \ref{slope}.\\
			\begin{figure*}[!h]
				\centering
				\includegraphics[scale=0.5]{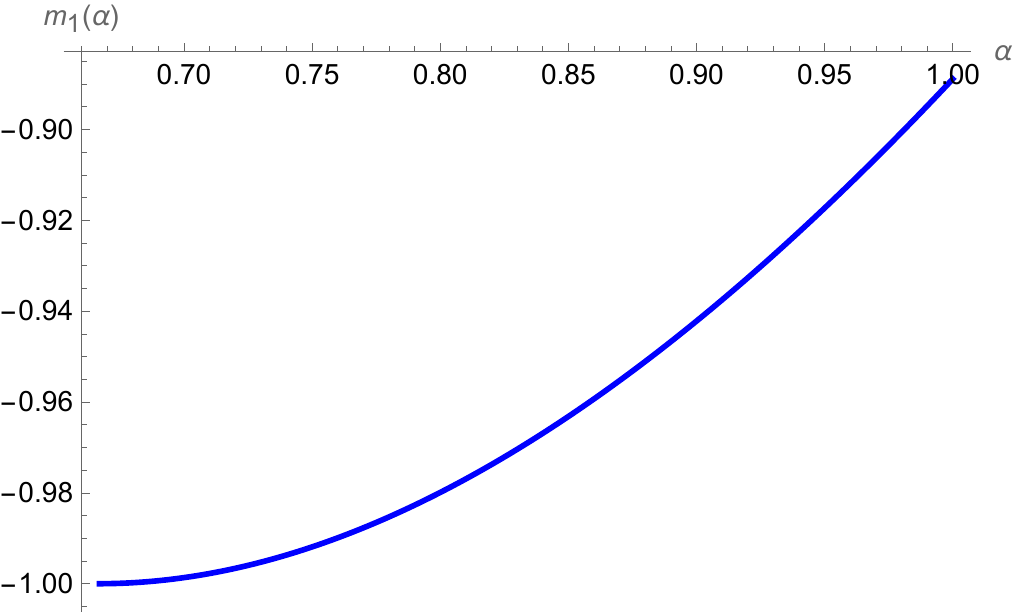}
				
				\caption{$\alpha \; \text{vs} \; m_1(\alpha)$} 
				\label{slope}
			\end{figure*}
			Also,	$m_1(\alpha)$ can be approximated as a polynomial of degree 4,i.e.,\\
			$m_1(\alpha)\approx 0.283115\alpha^4-1.53076\alpha^3+3.53266\alpha^2-3.00438\alpha-0.16952.$
			
			\item unstable $\forall \; \tau \ge 0 \; \text{if} \; a>0.$
			\item asymptotically stable for $0\le\tau<\tau^*$ and unstable for $\tau>\tau^*$, leading to single stable region (SSR), if $a<0$ and $b\in(-\infty,a/2)\cup(-a,\infty)$, where the critical value $\tau^*$ is described as below:\\
			If $b\in(-a,\infty)$ i.e., $(a,b)$ belongs to the second quadrant then,
			\begin{itemize}
				\item  If $\alpha\le0.2644385$ then critical value $\tau^*(\Gamma_2)$ of delay is given by the intersection of $\bar{V} \; \text{with} \; \Gamma_2$.
				\item If $\alpha>0.2644385$ then we find the point $(A_1(\alpha),B_1(\alpha))$ (see data set 1).\\
				--  If $A<A_1(\alpha)$ then $\tau^*(\Gamma_2)$ will give the critical value of delay. \\
				-- If $A>A_1(\alpha)$ then the critical delay value will be $\tau^*(\Gamma_1)$.
			\end{itemize}
			If $b\in(-\infty,a/2)$ i.e., $(a,b)$ belongs to the third quadrant, then the critical value will be given by $\tau^*(\Gamma_3)$.\\
			The values $\tau^*(\Gamma_j),\; j\in\{1,2,3\}$ are given in data set 2.
			
			\item asymptotically stable for $0<\tau<\tau_1^*$, unstable for $\tau_1^*<\tau<\tau_2^*$ and again asymptotically stable for 
			$\tau>\tau_2^*$, leading to the stability switch (SS), if $2/3<\alpha\le1$, $a<0$ and $m_1(\alpha)a<b<-a$, where $m_1(\alpha)$ is as in case (1) and the constants $\tau_1^*,\;\tau_2^*$ are mentioned in the data set 2.
		\end{enumerate}
	\end{The}

\begin{figure}[H]
	\centering
	
	\begin{subfigure}[b]{0.9\textwidth}
		\centering
		\includegraphics[width=\textwidth]{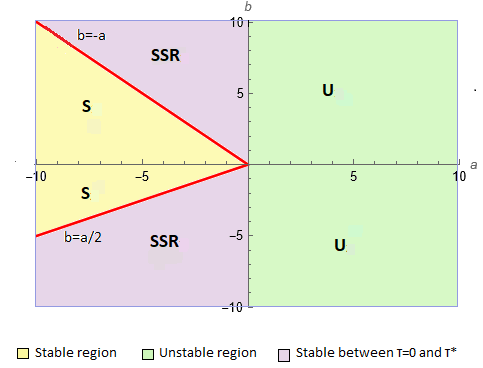} % Replace with your image
		\caption{$\alpha=0.2$}
		\label{fig:ab1}
	\end{subfigure}
	\hspace{0.2cm}
	\begin{subfigure}[b]{0.85\textwidth}
		\centering
		\includegraphics[width=\textwidth]{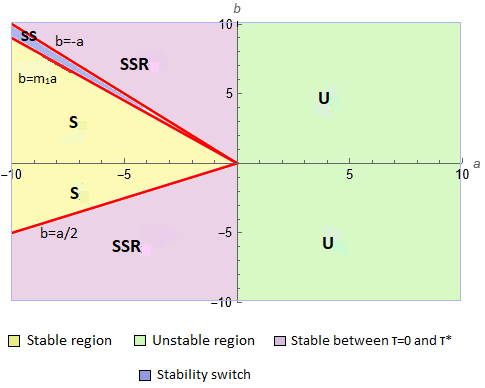} % Replace with your image
		\caption{$\alpha=0.5$}
		\label{fig:ab2}
	\end{subfigure}

	\caption{Stability diagrams in $ab$-plane}
	\label{fig:Fig. 5}
	
\end{figure}

	\newpage
	
	\section{Illustrative examples}
	\label{IE}
	\textbf{ Example 1 } Consider the FDDE (\ref{eq.1}) with $\alpha=0.4$. We take different values of $a,b\; \text{and} \; \tau$ and verify the stability regions given in Figure \ref{fig:ab1}.
	
	(i)
	Consider $ a=5,\; b=7\; \text{and}\; \tau=0.4.$\\
	Since both $a,b>0 $. So, $(a,b)\in U$ region.\\
	Hence, the system is unstable for any $\tau$.
	We sketched an unbounded solution for this case in Figure \ref{Fig. 7}. 
	\begin{figure}[H]
		\centering
		\includegraphics[scale=0.6]{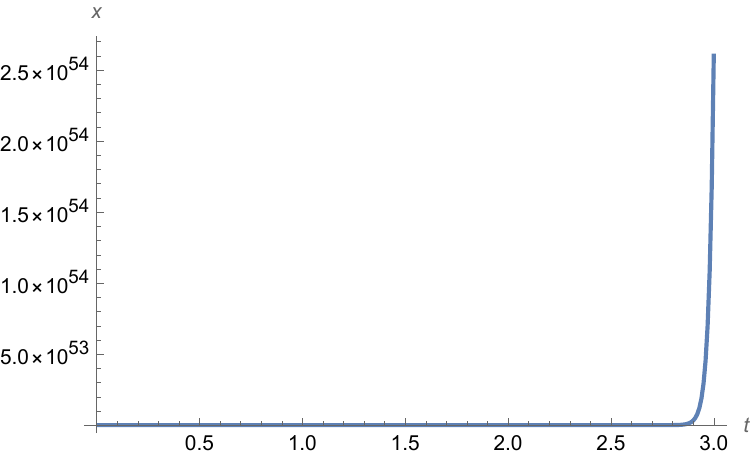}  
		\caption{$\alpha=0.4, a=5,b=7, \tau=0.4$, unstable solution}
		\label{Fig. 7}
	\end{figure}

	(ii)
	Now, consider $a=-3,\; b=5 \; \text{and} \; \tau=0.25$.\\
	Since\; $ b>-a,\;(a,b)\in$ SSR.
	Furthermore, $\alpha>\alpha^*=0.2644385$, thus, based on theorem \ref{Main Theorem}, we find the intersection point and get
	$A_1(0.4)=-1.1541$ (see data set 1).
	Now, $A=a\tau^\alpha= (-3){(0.25)}^{0.4}=-1.723<A_1.$
	So, $\tau^*=\tau^*(\Gamma_2)=0.271611$ (see  data set 2).\\
	Since $\tau<\tau^*$, we get the stable solution shown in Figure \ref{Fig. 8}.
	\begin{figure}[H]
		\centering
		\includegraphics[scale=0.8]{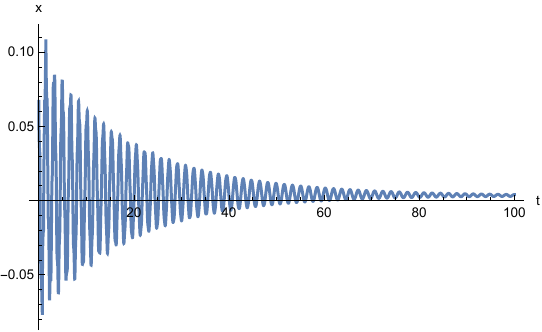}  
		\caption{{$\alpha=0.4, a=-3,b=5, \tau=0.25$, stable solution}}
		\label{Fig. 8}
	\end{figure}
	Take the same values of $a$ and $b$ with $\tau=0.28$.
	In this case, $A=(-3){(0.28)}^{0.4}=-1.80295<A_1.$ 
	So, $\tau^*=\tau^*(\Gamma_2)=0.271611$ (see  data set 2).\\
	Since $\tau>\tau^*$, the system is unstable. We get a root $\beta=0.00287472 + 5.02187i$ with positive real part of the modified characteristic equation (\ref{mce}) supporting this claim.
	\vspace{0.1cm}	
	
	(iii)
	Take $a=-5,\; b=3\; \text{and} \; \tau=0.8$.\\
	Since\;\; $ b<-a.$
	So, $(a,b)\in S.$\\
	Hence, we get the stable solution shown in Figure (\ref{Fig. 9}). 
	\begin{figure}[H]
		\centering
		\includegraphics[scale=0.6]{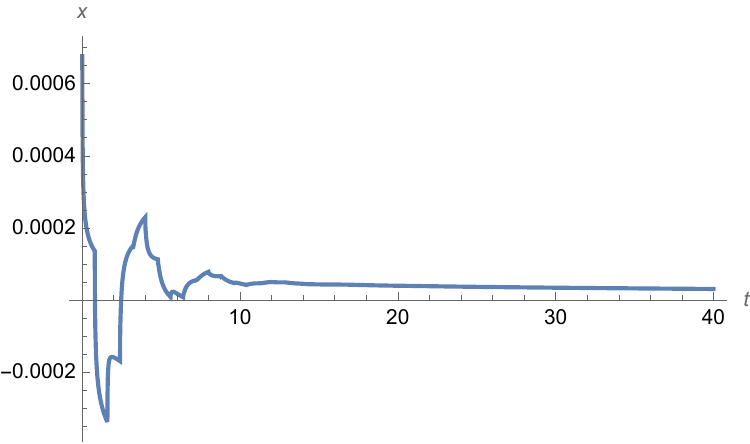} 
		\caption{{$\alpha=0.4, a=-5,b=3, \tau=0.8$, stable solution}}
		\label{Fig. 9}
	\end{figure}      
	\vspace{0.1cm}		
	
	(iv) 
	Now, consider $a=-3\;, b=-4$ and $\tau=0.02$.\\
	Since\;\; $ b<a/2$. So, $(a,b)\in$ SSR.
	In this case, $\tau^*=\tau^*(\Gamma_3)=0.0443753$ (see data set 2). \\
	Hence, the system is stable for $\tau=0.02<\tau^*$. We sketched the stable solutions as shown in Figure (\ref{Fig. 10}).
	\begin{figure}[H]
		\centering
		\includegraphics[scale=0.6]{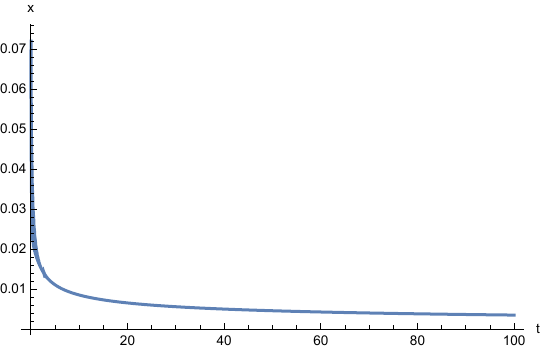} 
		\caption{{$\alpha=0.4, a=-3,b=-4, \tau=0.02$, stable solution}}
		\label{Fig. 10}
	\end{figure}
	On the other hand, if $\tau=0.05>\tau^*$, then the system is unstable (see Figure \ref{Fig. 11}).
	\begin{figure}[H]
		\centering
		\includegraphics[scale=0.6]{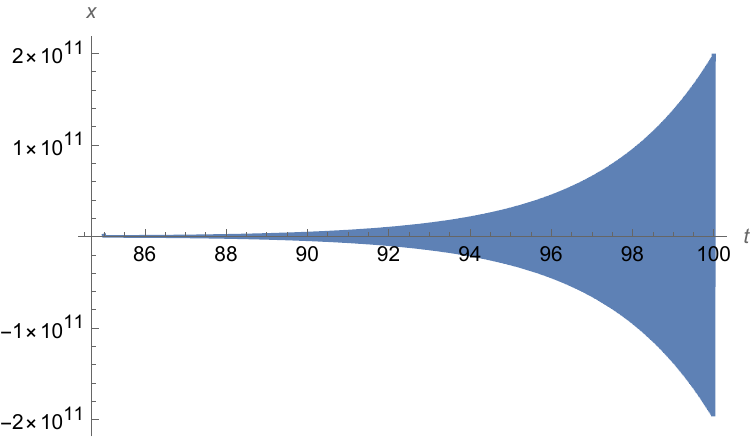}  
		\caption{{$\alpha=0.4, a=-3,b=-4, \tau=0.05$, unstable solution}}
		\label{Fig. 11}
	\end{figure}
	
	\newpage
	\textbf{Example 2}  Consider the FDDE (\ref{eq.1}) with $\alpha=0.8$. In this case, the stability analysis is provided by Figure \ref{fig:ab2}.

	(i)\;\; If we consider $a=5,\; b=-8 \; \text{and}\; \tau=1.5.$
	Then,  $a>0 \;\;and \;\;b<0 $.
	So, $(a,b)\in U$ region.\\
	Hence, the system is unstable for any $\tau$. We get an unstable solution shown in Figure \ref{Fig. 13}.
	\begin{figure}[H]
		\centering
		\includegraphics[scale=0.6]{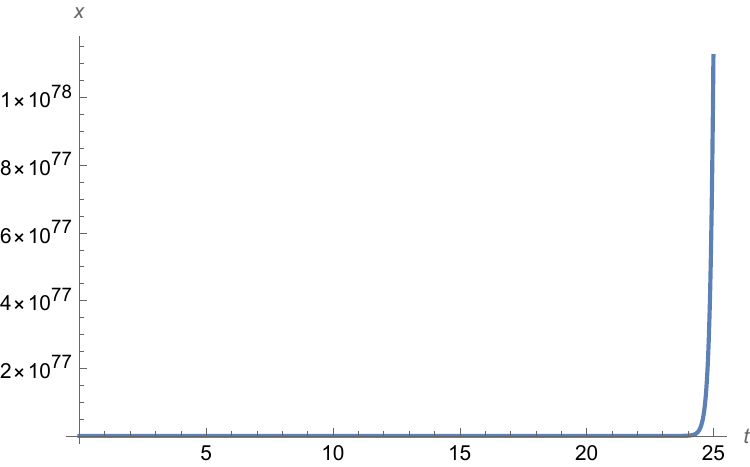}
		\caption{{$\alpha=0.8, a=5,b=-8, \tau=1.5$, unstable solution}} 
		\label{Fig. 13}
	\end{figure}
	
	(ii)
	Now, take $a=-5,\; b=8\; \text{and} \; \tau=0.25$.\\
	Here $ b>-a$, so, $(a,b)$ lies in the SSR region.
	Moreover, $\alpha>\alpha^*=0.2644385$, so following Theorem 4.1, we determine the intersection point and get $A_1(0.8)=-2.42855$ (see data set 1).
	Now, $A=a\tau^\alpha= (-5){(0.25)}^{0.8}=-1.64938>A_1.$
	So, $\tau^*=\tau^*(\Gamma_1)=0.264173$ (see  data set 2).\\
	Since $\tau<\tau^{*}$, the system will be stable (cf. Figure \ref{Fig. 14}).
	\begin{figure}[H]
		\centering
		\includegraphics[scale=0.7]{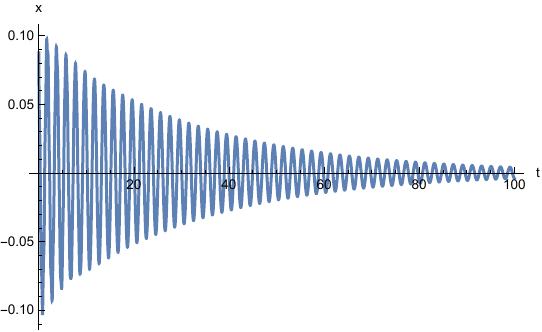}  
		\caption{{$\alpha=0.8, a=-5,b=8, \tau=0.25$, stable solution}}
		\label{Fig. 14}
	\end{figure}
	Take the same values of $a$ and $b$ with $\tau=0.28$.
	In this case, $A=(-5){(0.28)}^{0.8}=-1.80591> A_1.$ 
	So, $ \tau^*=\tau^*(\Gamma_1)=0.264173 $ (see  data set 2).\\
	Since $ \tau>\tau^* $, the system is unstable. We get an unstable solution as shown in Figure \ref{Fig. 15}.
	\begin{figure}[H]
		\centering
		\includegraphics[scale=0.7]{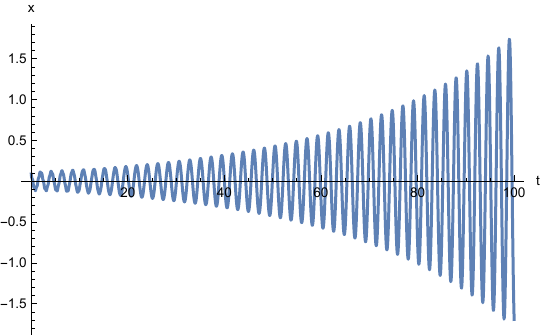} 
		\caption{{$\alpha=0.8, a=-5,b=8, \tau=0.28$, unstable solution}}
		\label{Fig. 15}
	\end{figure}

	(iii)
	If we assume $a=-7,\;b=4 \;\text{and}\; \tau=1.4$.\\
	Here $b/a=-0.71$ and equation of $T_1$ for $\alpha=0.8$ is $b=-0.9799a.$ Clearly, the line representing the given value of (a,b) will lie below $T_1$, i.e., in the stability region.\\	

	Hence, the system is stable for any $\tau$. Stable solutions are sketched in Figure \ref{Fig. 16}.
	\begin{figure}[H]
		\centering
		\includegraphics[scale=0.6]{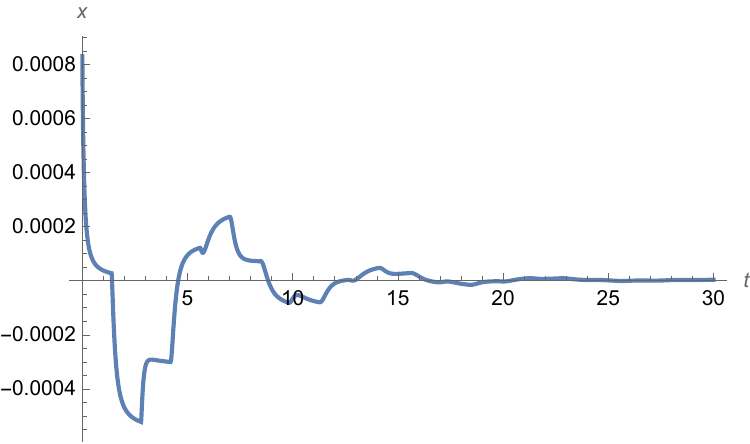}  
		\caption{{$\alpha=0.8, a=-7,b=4, \tau=1.4$, stable solution}}
		\label{Fig. 16}
	\end{figure}  
	
	\vspace{0.1cm}
	
	(iv)
	Now, we consider $a=-8,\;b=-6 \;\text{and}\; \tau=0.24$.\\
	Since\;\; $ b<a/2.$	So, $(a,b)\in SSR.$
	In this case, $\tau^*=\tau^*(\Gamma_3)=0.264942$ \\
	Hence, the system is stable for $\tau=0.24<\tau^*$ (cf. Figure \ref{Fig. 17}).
	\begin{figure}[H]
		\centering
		\includegraphics[scale=0.6]{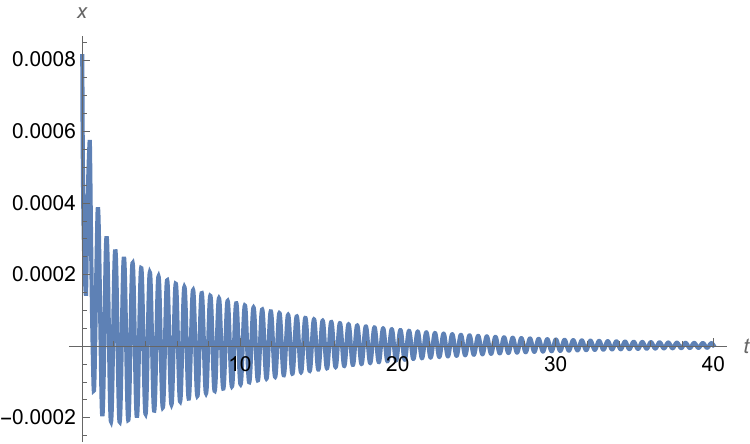}  
		\caption{{$\alpha=0.8, a=-8,b=-6, \tau=0.24$, stable solution}}
		\label{Fig. 17}
	\end{figure}

	If we consider the same values for $a$ and $b$ and $\tau=0.3$.\\ Then, since $\tau>\tau^*$. The solution trajectory shows unbounded oscillations as shown in Figure \ref{Fig. 18}.
	
	\begin{figure}[H]
		\centering
		\includegraphics[scale=0.6]{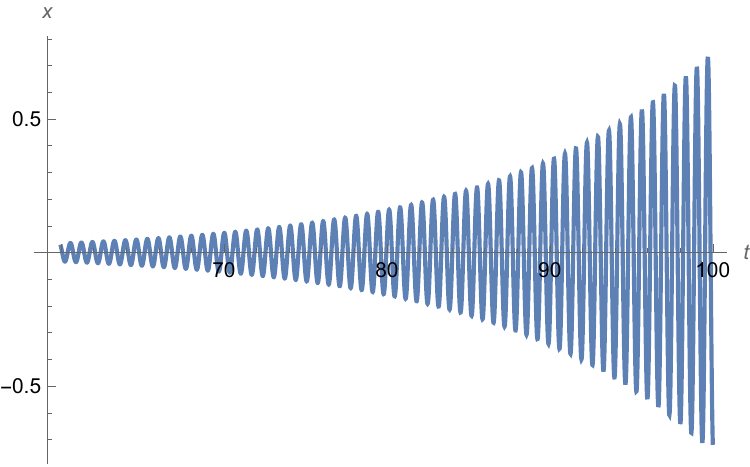}  
		\caption{{$\alpha=0.8, a=-8,b=-6, \tau=0.3$, unstable solution}}
		\label{Fig. 18}
	\end{figure}  
	
	(v)
	If we take $a=-30,\; b=29.4\; \text{and}\; \tau=0.2$.\\
	Here $b/a=-0.98$, equation of $T_1$ is $b=-0.9799a$ and equation of $T_2$ is $b=-a$. Hence, the line representing the given value of (a,b) will lie between $T_1$ and $T_2$ i.e., in the region of the stability switch.\\
	So, by putting values of $\alpha, a,b $, we will get both the critical values of this region.\\
	$\tau_1^*=0.576569$ and $\tau_2^*=0.693395$.
	
	So, for $\tau=0.2<\tau_1^*$, system is stable. Hence, the solution goes to zero as we increase t (see Figure \ref{Fig. 19}).
	\begin{figure}[H]
		\centering
		\includegraphics[scale=0.6]{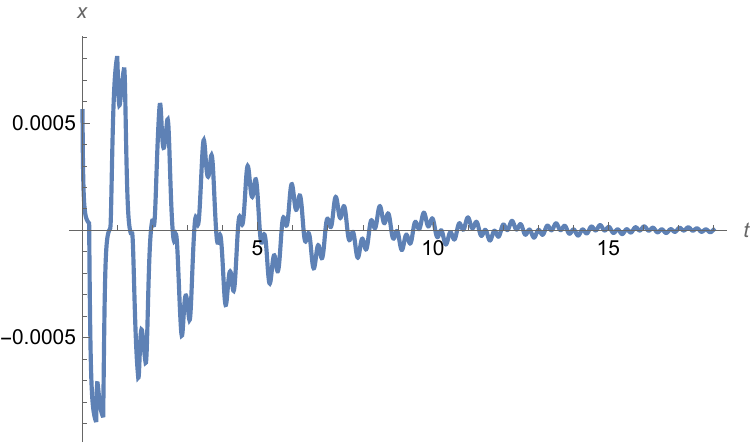}  
		\caption{{$\alpha=0.8, a=-30,b=29.4, \tau=0.2$, stable solution}}
		\label{Fig. 19}
	\end{figure}     
	
	Now, if $\tau=0.6$ i.e $\tau_1^*<\tau<\tau_2^*$, then system will be unstable.
	We observed that there is a positive root $\beta=0.614969$ of the modified characteristic equation (\ref{mce}) validating this assertion.\\
	
	Also, If we take $a=-30,\; b=29.4\; \text{and} \;\tau=0.8>\tau_2^*$, then the system is again stable(see Figure \ref{Fig. 21}).
	\begin{figure}[H]
		\centering
		\includegraphics[scale=0.8]{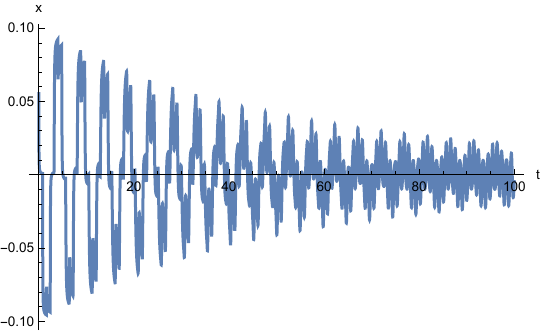}  
		\caption{{$\alpha=0.8, a=-30,b=29.4, \tau=0.8$, stable solution}}
		\label{Fig. 21}
	\end{figure} 
	\section{Conclusions}
	\label{con}
	We considered the fractional order delay differential equation $D^\alpha x(t)=a x(t)+b x(t-\tau)-b x(t-2\tau)$. We reduced the number of parameters by transforming $A=a\tau^\alpha$,  $B=b\tau^\alpha$. The boundaries of the stable region in the $AB-$ plane are obtained by setting the eigenvalue as a purely imaginary value. This provided stable and unstable regions. Furthermore, we translated these regions to $ab-$plane. This generated a few more delay-dependent regions, viz. single stable region (SSR), where the system is stable for smaller values of delay and becomes unstable for the larger ones, and stability switch (SS), where we get the intermittent stability behavior as the delay changes. We provided an ample number of examples to support our results. We hope that this work will be an essential step to solve the open problem on the stability analysis of $D^\alpha x(t)=a x(t)+b x(t-\tau_1)+c x(t-\tau_2)$. 
	\section*{Acknowledgment}	
	Pragati Dutta thanks the University of Hyderabad for the non-net fellowship.

	\newpage
	\bibliographystyle{unsrt}	
	\bibliography{ref.bib}
	
\end{document}